\documentclass[12pt,oneside]{amsart}
\usepackage{amsmath}
\usepackage{amsthm}
\usepackage{amsfonts}
\usepackage{amssymb}
\usepackage{enumerate}
\newtheorem{theorem}{Theorem}

\newtheorem{proposition}[theorem]{Proposition}
\newtheorem{corollary}[theorem]{Corollary}

\newtheorem{prop}{Proposition}
\newtheorem{thm}{Theorem}

\theoremstyle{definition}

\newtheorem{example}[theorem]{Example}

\theoremstyle{remark}
\newtheorem{remark}[theorem]{Remark}
\newcommand{\DD}{{\mathbb D}}

\DeclareMathOperator{\Aut}{Aut} 
 
 \DeclareMathOperator{\re}{Re}

\DeclareMathOperator{\im}{Im}
\renewcommand{\phi}{\varphi}


\begin{document}

\title[Suita conjeture and problem of Wiegerinck]{Generalizations of the higher dimensional Suita conjecture and its relation with a problem of Wiegerinck}

%\address{Carl von Ossietzky Universitat Oldenburg, Institut fur Mathematik, Postfach 2503, ¨
%D-26111 Oldenburg, Germany}

\address{Institute of Mathematics, Faculty of Mathematics and Computer Science, Jagiellonian
University,  \L ojasiewicza 6, 30-348 Krak\'ow, Poland}

\author{Zbigniew B\l ocki}\email{Zbigniew.Blocki@uj.edu.pl}
\author{W\l odzimierz Zwonek}\email{Wlodzimierz.Zwonek@uj.edu.pl}

\thanks{The first author was supported by the Ideas Plus grant 0001/ID3/2014/63 of the Polish Ministry of Science and Higher Education and the second named author by the OPUS grant no. 2015/17/B/ST1/00996 financed by the National Science Centre, Poland.}


\begin{abstract} We generalize the inequality being a counterpart of the several complex variables version of the Suita conjecture. For this aim higher order generalizations of the Bergman kernel are introduced. As a corollary some new partial results on the dimension of the Bergman space in pseudoconvex domains are given. A relation between the problem of Wiegerinck on possible dimension of the Bergman space of unbounded pseudoconvex domains in general case and in the case of balanced domains is also shown. Moreover, some classes of domains where the answer to the problem of Wiegerinck is positive are given. Additionally, regularity properties of functions involving the volumes of Azukawa indicatrices are shown.

\end{abstract}
\maketitle

\section{Introduction} Recall that for the domain $D\subset\mathbb C^n$, $w\in D$ we define \textit{the Bergman kernel $K_D$} as follows
\begin{equation}
K_D(w):=\sup\left\{|f(w)|^2:f\in\mathcal O(D),\
     ||f||_D^2:=\int_D|f|^2d\lambda^{2n}\leq 1\right\}.
\end{equation}
We put $L_h^2(D):=L^2(D)\cap\mathcal O(D)$.

Additionally, if $K_D(z)>0$ then we denote by $\beta_D$  \textit{ the Bergman metric } induced by the Bergman kernel:
\begin{equation}
\beta_D(z;X):=\sqrt{\sum_{j,k=1}^n\frac{\partial^2\log K_D(z)}{\partial z_j\bar z_k}X_j\overline{X_k}},\;  X\in\mathbb C^n.
\end{equation}

We also define \textit{the Azukawa pseudometric} as follows
\begin{equation}
A_D(w;X):=\exp\left(\limsup_{\lambda\to 0}(G_{D}(w+\lambda X,w)-\log|\lambda|)\right),
\end{equation}
$w\in D$, $X\in\mathbb C^n$, 
where $G_D(\cdot,w)=G_w(\cdot)$ denotes \textit{the pluricomplex Green function with the pole at $w$}.

Denote also \textit{the Azukawa indicatrix at $w$}:
\begin{equation}
I_D(w):=\{X\in\mathbb C^n:A_D(w;X)<1\}.
\end{equation}

Recall that a recently obtained version of the higher dimensional version of the Suita conjecture (see \cite{Blo-Zwo 2015}) 
\begin{equation}\label{equation:suita-conjecture}
K_D(w)\geq\frac{1}{\lambda^{2n}(I_D(w))},\; w\in D
\end{equation}
which holds for any pseudoconvex domain may be formulated as follows
\begin{equation}
K_D(w)\geq K_{I_D(w)}(0),\; w\in D.
\end{equation}
Making use of the reasoning as in \cite{Blo 2014a}, \cite{Blo 2014b} and \cite{Blo-Zwo 2015} we generalize this inequality (see Theorem~\ref{theorem:non-decreasing}) which then may be applied to getting positive results on non-triviality of the Bergman space and its infinite dimensionality (see Section~\ref{section:suita-dimension}). Thus it gives a partial solution to a problem of Wiegerinck on the possible dimension of the Bergman space on unbounded pseudoconvex domains (see the problem posed in \cite{Wie 1984}). 

The generalization of the Suita conjecture requires the definition of the higher order Bergman kernels. The introduced objects as well as analoguous inequalities have been recently presented in the case of one dimensional domains in the  paper \cite{Blo-Zwo 2018}. 

In Section~\ref{section:wiegerinck-positive} we present other classes of domains where the problem of Wiegerinck is solved positively.

In our paper we also present some results that are motivated by the objects that were introduced and studied in the paper \cite{Blo-Zwo 2015}; in particular, in Section~\ref{section:regularity} we show regularity properties of the volume of the Azukawa indicatrix.

\section{Higher dimensional generalization of the Suita conjecture}
Let $H$ be a homogeneuous polynomial on $\mathbb C^n$ of degree $k$, $H(z)=\sum_{|\alpha|=k}a_{\alpha}z^{\alpha}$. 
We may well define the operator
\begin{equation}
P_{H}(f):=\sum_{|\alpha|=k}a_{\alpha}D^{\alpha}f,
\end{equation}
where $f\in\mathcal O(D)$ for some domain $D\subset\mathbb C^n$.

For the fixed domain $D\subset\mathbb C^n$, $z\in D$  we define
\begin{multline}
K_D^H(z):=\\
\sup\{|P_{H}(f)(z)|^2:f^{(j)}(z)=0,\; j=0,\ldots,k-1,\; f\in L_h^2(D), ||f||_D\leq 1\}.
\end{multline}
$f^{(j)}(z)$ denotes the $j$-th Frechet derivative of $f$ at $z$ - it is meant here as a homogeneuous polynomial of degree $j$.

Note that
\begin{equation}
K_D^{1}(z)=K_D(z).
\end{equation}
For $X\in\mathbb C^n$ put $H_X(z):=X_1z_1+\ldots+X_nz_n$. If $K_D(z)>0$ then
\begin{equation}
\beta_D^2(w;X)=\frac{K_D^{H_X}(w)}{K_D^{1}(w)}.
\end{equation}
We also put
\begin{multline}
K_D^{(k)}(w;X):=K_D^{H_X^k}(w)=\\
\sup\{|f^{(k)}(w)(X)|^2:f\in L_h^2(D),\; f^{(j)}(w)=0,\;j=0,\ldots,k-1,\; ||f||_D\leq 1\}.
\end{multline} 
Note that in the case $n=1$ we have $K_D^{(k)}(z;1)=K_D^{(k)}(z)$, where the expression on the right side is understood as in the paper \cite{Blo-Zwo 2018}. 

Following the proof of the analoguous result in the case of the Bergman kernel we get the following fundamental properties of $K_D^{H}$.

\begin{proposition}\label{proposition:basic-properties}
\begin{itemize}

\item Let $F:D\to G$ be a biholomorphic mapping, and let $H$ be a homogeneuous polynomial of degree $k\in\mathbb N$, $w\in D$. Then
\begin{equation}
K_G^{H}(F(w))=K_D^{H\circ F^{\prime}(w)}(w)|\det F^{\prime}(w)|^{2}.
\end{equation}

\item Let $D_1,\ldots,D_m$ be domains in $\mathbb C^n$, $w^j\in D_j$, and let $H^j$ be a homogeneuous polynomial on $\mathbb C^n$. Then
\begin{equation}
K_{D_1\times\ldots\times D_m}^{H_1\times\ldots\times H_m}(w^1,\ldots,w^m)=K_{D_1}^{H_1}(w^1)\cdot\ldots\cdot K_{D_m}^{H_m}(w^m).
\end{equation}

\item If $D$ is a balanced pseudoconvex domain, $H$ is a homogeneuous polynomial on $\mathbb C^n$ then
\begin{equation}
K_D^H(0)=\frac{|P_{D,H}(H^*)|^2}{||H||_{D}^2}=\frac{\left|\sum_{|\alpha|=k}|a_{\alpha}|^2\alpha!\right|^2}{\int_D|H(z)|^2d\lambda^{2n}(z)},
\end{equation}
where $H^*(z)=\sum_{|\alpha|=k}\bar a_{\alpha}z^{\alpha}$.
\end{itemize}

\end{proposition}

To make the presentation simpler we shall often assume that the point (pole of the Green function) will be $w=0$.
 In such a case we denote
$D_a:=e^{-a}\{G<a\}$ for $a\leq 0$. Additionally, put $D_{-\infty}:=I_D(0)$. We shall often use the obvious fact that the sets
$\{G<a\}$ and $D_a$ are linearly isomorphic, $-\infty<a\leq 0$.

The properties of the Green function give the equality $(D_a)_b=D_{a+b}$ for $-\infty\leq a,b\leq 0$. Note also that $K_{D_a}^{H}(0)=e^{2(n+k)a}
K_{\{G<a\}}^{H}(0)$.

Our main result is the following.

\begin{theorem}\label{theorem:non-decreasing} Let $D$ be a pseudoconvex domain in $\mathbb C^n$, $w=0\in D$ and let $H$ be a homogeneuous polynomial of degree $k$. Then the function 
\begin{equation}
[-\infty,0]\owns a \to K_{D_a}^{H}(0)  
\end{equation}
is non-decreasing and continuous at $a=-\infty$. In particular, $K_{I_D(0)}^H(0)\leq K_D^H(0)$ so $K_{I_D(0)}^{(k)}(0;X)\leq K_{D}^{(k)}(0;X)$ for any $X\in\mathbb C^n$, too.
\end{theorem}
\begin{proof} We compile the reasoning as in the proof of Theorem 1 in \cite{Blo 2014a}, the proof of Theorem 6.3 in \cite{Blo 2014b} and the proof of Theorem 1 in \cite{Blo-Zwo 2015}. 

Standard Ramadanov type reasoning allows us to reduce the situation to the case when $D$ is a bounded hyperconvex domain.

Let us take $a\leq 0$. First we show the monotonicity of the function on the interval $(-\infty,0]$. 

The properties of the Green function and thus the ones of the sets $D_a$ reduce the problem of the monotonicity of the function for $a>-\infty$ to the proof of the inequality $K_D^{H}(0)\geq K_{D_a}^{H}(0)$.

Take any $f\in L_h^2(\{G<a\})$ with $f^{(j)}(0)=0$, $j=0,\ldots,k-1$. We use a theorem of Donelly-Feffermann (see \cite{Don-Fef 1983} or Theorem 2.2 in \cite{Blo 2014b}) with the following data
\begin{equation}
\phi:=2(n+k+1)G,\;\psi:=-\log(-G),\;\alpha:=\bar\partial (f\chi\circ G),
\end{equation}
where 
\begin{equation}
\chi(t):=
\begin{cases}
0,& t\geq a\\
\int_{-a}^{-t}\frac{e^{-(n+k+1)s}}{s}ds,& t<a
\end{cases}.
\end{equation}
We take $u$ with $\bar\partial u=\alpha$ with the estimate as in \cite{Blo 2014a}.

Then the holomorphic function
\begin{equation}
F:=f\chi\circ G-u
\end{equation}
satisfies $F^{(j)}(0)=0$ and $P_{D,H}(F)(0)=\chi(-\infty)P_{D,H}(f)(0)=Ei(-(n+k+1)a)P_{D,H}(f)(0)$. Moreover,
\begin{equation}
||F||_{L^2(D)}\leq (\chi(-\infty)+\sqrt{C})||f||_{\{G<a\}},
\end{equation}
which implies that
\begin{equation}
K_D^{H}(0)\geq c(n,a,k)K_{\{G<a\}}^{H}(0),
\end{equation}
where $c(n,a,k)=\frac{Ei(-(n+k+1)a)^2}{(Ei(-(n+k+1)a)+\sqrt{C})^2}$.

The tensor power trick together with Proposition~\ref{proposition:basic-properties} gives the inequality
\begin{equation}
K_D^{H}(0)\geq e^{2(n+k)a}K_{\{G<a\}}^{H}(0)=K_{D_a}^{H}(0).
\end{equation} 
Similarly, as in \cite{Blo-Zwo 2015} we note that the continuity of the Azukawa metric (and the existence of the limit in its definition) - see \cite{Zwo 2000a} and \cite{Zwo 2000b} - implies the convergence in the sense of Hausdorff:
$D_a\to I_D(w)$
which together with basic properties of the Bergman functions implies the desired inequality on $[-\infty,0]$ and the continuity at $-\infty$.
\end{proof}

%\begin{remark}
%It easily follows from the above proof and the properties of the sublevel sets of the Green function that the function
%\begin{equation}
%[0,\infty)\owns a \to e^{2n(k+1)a}K_{\{G<-a\}}^{(k)}(w;X)  
%\end{equation}
%is non-increasing.

%\end{remark}

\begin{remark} It would be interesting to verify whether the function
\begin{equation}
(-\infty,0]\owns a\to \log K_{D_a}^{H}(w)
\end{equation}
is convex as it is in the case of $H\equiv 1$ (see final remark in \cite{Blo 2015})?
\end{remark}

Note that the nontriviality of the space $L_h^2(D)$ is equivalent to the fact that for any $w\in D$ there are a $k$ and $X$ such that $K_D^{(k)}(w;X)>0$. 

The infinite dimensionality of $L_h^2(D)$ is equivalent to the existence for any (equivalently, some) $w\in D$ a subsequence $(k_{\nu})$ and a sequence $(X^{\nu})$ such that $K_D^{(k_{\nu})}(w;X^{\nu})>0$. Therefore, we conclude

\begin{proposition}\label{proposition:first-implication} Let $D$ be a pseudoconvex domain in $\mathbb C^n$. 
\begin{itemize}
\item If for some $w\in D$ the space $L_h^2(I_D(w))$ is not trivial then so is the space $L_h^2(D)$.
\item
If for some $w\in D$ the dimension of $L_h^2(I_D(w))$ is infinite then so is the dimension of $L_h^2(D)$.

\end{itemize}

\end{proposition}

%It is obvious that the set $\{G<a\}$ and $D_a$ are biholomorphic.

In fact, one may also conclude from Theorem~\ref{theorem:non-decreasing} a more precise version of Proposition~\ref{proposition:first-implication}.

\begin{corollary}\label{corollary:dimension-inequality}  Let $D$ be a pseudoconvex domain in $\mathbb C^n$, $w\in D$, $-\infty<a\leq 0$. Then
\begin{equation}
\operatorname{dim}(L_h^2(I_D(w)))\leq \operatorname{dim}(L_h^2(D_a(w))).
\end{equation}
\end{corollary}

Making use of the result from \cite{Pfl-Zwo 2017} we get the following partial solution of the problem of Wiegerinck (see \cite{Wie 1984}).

\begin{corollary} Let $D$ be a pseudoconvex domain in $\mathbb C^2$. If for some $w\in D$ the space $L_h^2(I_D(w))$ is not trivial then the dimension of $L_h^2(D)$ is infinite.
\end{corollary}

Note that  the non-triviality of the space $L_h^2(I_D(w))$ in the case $n=2$ is precisely described in \cite{Pfl-Zwo 2017}.

\begin{remark} To answer the problem of Wiegerinck in dimension two it would be then sufficient to decide what the dimensions of $L_h^2(D)$ are in the case when $L_h^2(I_D(w))=\{0\}$ for all $w\in D$. 
The solution of that problem seems to be very probable to get. Perhaps one should start with the solution of the problem when $A_D\equiv 0$, or $G\equiv -\infty$?
\end{remark}

\section{On the finite-dimensional Bergman space on $D_a$}\label{section:suita-dimension}

%Consider at first the situation when $D$ is a bounded pseudoconvex domain. Fix the pole $w=0\in D$ and put $D_a:=e^{-a}\{G<a\}$ for $a\leq 0$. %Additionally put $D_{-\infty}:=I_D(0)$.

%The properties of the Green function give the equality $(D_a)_b=D_{a+b}$ for $-\infty\leq a,b\leq 0$. 

Note that Corollary~\ref{corollary:dimension-inequality} leaves the problem on the mutual relation between the dimensions of the spaces 
$L_h^2(D_a)$ for different $a$ open. Note that the restriction: $L_h^2(D_b)\owns f\to f(e^{a-b}\cdot)_{|D_a}\in L_h^2(D_a)$, $-\infty<a<b\leq 0$ gives the inequality
\begin{equation}
\operatorname{dim}(L_h^2(D_a)\leq \operatorname{dim}(L_h^2(D_b)).
\end{equation}
 In fact, we shall prove that the equality holds.

\begin{proposition}\label{proposition:dimension-equality} Let $D$ be a pseudoconvex domain in $\mathbb C^n$, $0\in D$. Then for any $-\infty<a\leq 0$ the dimension of $L_h^2(D_a)$ is the same.
\end{proposition}
\begin{proof} It is sufficient to show that if $-\infty<a<0$ then the dimension of $L_h^2(D_a)$ is equal to that of $L_h^2(D)$. To prove this it is sufficient to show that if we get the system $\{f_1,\ldots,f_N\}$ of linearly independent elements of $L_h^2(D_a)$ then there are elements $F_1,\ldots,F_N$ from $L_h^2(D)$ linearly independent. For the functions $f_l$ we follow a construction from the proof of Theorem~\ref{theorem:non-decreasing}. First we choose $k$ so big that
the functions $\tilde f_l$, $l=1,\ldots,N$, are linearly indpendent in the space of polynomials, where $\tilde f_l(z):=\sum_{l=0}^k\frac{f_j^{(l)}(0)}{l!}(z)$, $j=1,\ldots,N$. Fix now a smooth  function $\chi:[-\infty,0]\to[0,1]$ such that $\chi$ equals $1$ near $-\infty$ and $\chi(t)=0$, $t\geq a$. Now starting with the functions $f_l$ we proceed with the construction of functions $F_l$ as in the proof of Theorem~\ref{theorem:non-decreasing} with $\varphi:=2(n+k+1)G$ and the mapping $\chi$. The functions $F_l$ are $L_h^2$ functions on $D$ that satisfy the equality
% being $1$ near $-\infty$ and $0$ for $t\geq a$) we define:
%\begin{equation}
%L_{D,a,k}:L_h^2(D_a)\owns f\to f(e^{a}\cdot)\cdot \chi\circ G-u\in L_h^2(D).
%\end{equation} 
$\tilde f_l\equiv \tilde F_l$, $l=1,\ldots,N$, which implies immediately the linear independence of $F_l$, $l=1,\ldots,N$.
\end{proof}

\begin{remark} Proposition~\ref{proposition:dimension-equality} together with Corollary~\ref{corollary:dimension-inequality} suggest that the equality of dimensions of all Bergman spaces $L_h^2(D_a)$, $-\infty\leq a\leq 0$ may hold, which in turn would reduce the problem of Wiegerinck from the general case to that in the class of pseudoconvex balanced domains (the set $D_{-\infty}$).
\end{remark}

\begin{remark} Note that the the results presented in this section imply that if $L_h^2(D)$ is finitely dimensional then all the functions lying in $L_h^2(\{G<a\})$, $-\infty<a<0$ are the restrictions of the functions from $L_h^2(D)$ -- this very special phenomenon is a fact which may serve as another hint that the problem of Wiegerinck should have a positive answer.

\end{remark}

\section{Other sufficient conditions for the positive solution of the problem of Wiegerinck}\label{section:wiegerinck-positive}
In this section we shall present two other sufficient conditions on domains that 
guarantee that the domain from the given class will give the positive answer to the problem of Wiegerinck.

\begin{theorem} Let $D$ be a pseudoconvex domain in 
$\mathbb C^n$ such that for some $w\in D$ and $a\leq 0$ 
the sublevel set $\{G_D(\cdot,w)<a\}$ does not satisfy the 
Liouville property, that is there exists a bounded nonconstant 
holomorphic function defined there. Then the Bergman space
$L^2_h(D)$ is either trivial or infinitely dimensional. 
\end{theorem}

\begin{proof} Assume that
there exists nonzero $f\in L^2_h(D)$. There exists
$k\geq 0$ such that $f^{(j)}(w)=0$ for  $j=0,1,\dots,k-1$ 
but $f^{(k)}(w)\neq 0$. We can also find $Q$ holomorphic and bounded 
in $\{G<a\}$, where $G=G_D(\cdot,w)$, and $m\geq 1$ 
such that $Q^{(j)}(w)=0$ for  $j=0,1,\dots,m-1$ but 
$Q^{(m)}(w)\neq 0$. For $l\geq 1$ define
  $$\alpha:=\bar\partial(Q^lf\chi\circ G)=
     Q^lf\chi'\circ G\,\bar\partial G,$$
where $\chi\in C^\infty(\mathbb R,\mathbb R)$ is such that
$\chi(t)=1$ for $t\leq b$ and $\chi(t)=0$ for $t\geq c$,
where $b$ and $c$ are such that $b<a<c<0$. Set
  $$\varphi:=2(n+k+lm)G,\ \ \ \psi:=-\log(-G),$$
then  
  $$i\bar\alpha\wedge\alpha
     \leq|Q|^{2l}|f|^2(\chi'\circ G)^2i\partial G\circ\bar\partial G
     \leq|Q|^{2l}|f|^2(\chi'\circ G)^2G^2 i\partial\bar\partial\psi$$
and by the Donnelly-Fefferman estimate there exists 
$u\in L^2_{loc}(D)$ with $\bar\partial u=\alpha$ and
 \begin{multline}
||u||^2
    \leq\int_D|u|^2e^{-\varphi}d\lambda
    \leq \\
4\int_D |Q|^{2l}|f|^2(\chi'\circ G)^2G^2
    e^{-2(n+k+lm)G}d\lambda
    \leq C||f||^2.
\end{multline}
Set $F=Q^lf\,\chi\circ G-u$. Then 
$F\in L^2_h(D)$ and $F^{(j)}(w)=0$ for $j=0,\dots,k+lm-1$,
but $F^{(k+lm)}(w)\neq 0$. Since $l$ is arbitrary, it follows that 
$L^2_h(D)$ is infinitely dimensional. 
\end{proof}

\begin{theorem} Let $D$ be a pseudoconvex domain in $\mathbb C^n$
and $w_j\in D$ an infinite sequence, not contained in any 
analytic subset of $D$, and such that for every $j\neq k$ there
exists $t<0$ such that $\{G_j<t\}\cap\{G_k<t\}=\emptyset$, where
$G_j:=G_D(\cdot,w_j)$. Then $L^2_h(D)$ is either trivial
or infinitely dimensional.
\end{theorem}

\begin{proof} Assume that $f\in L^2_h(D)$, $f\not\equiv 0$.
Choosing a subsequence if necessary we may assume that 
$f(w_j)\neq 0$ for all $j$. For every $k$ we want to construct
$F\in L^2_h(D)$ such that $F(w_j)=0$ for $j=1,\dots,k-1$ but
$F(w_k)\neq 0$. It will then follow that $L^2_h(D)$ is
infitely dimensional.

We can find $t_k<0$ such that $\{G_j<t_k\}\cap\{G_l<t_k\}=\emptyset$
for $j,l=1,\dots,k$, $j\neq l$. Set $G:=G_1+\dots+G_{k-1}$ and 
  $$\alpha:=\bar\partial\big(f\chi\circ G\big)
    =f\chi'\circ G\,\bar\partial G,$$
where $\chi\in C^\infty(\mathbb R)$ is such that $\chi(t)=0$ for
$t\leq (k-1)t_k-2$ and $\chi(t)=0$ for $t\geq (k-1)t_k-1$.
Define the weights
  $$\varphi:=2n(G+G_k),\ \ \ \psi:=-\log(-G),$$
we then have
  $$i\bar\alpha\wedge\alpha\leq 
      |f|^2G^2(\chi'\circ G)^2i\partial\bar\partial\psi.$$
By the Donnelly-Fefferman estimate we can find 
$u\in L^2_{loc}(D)$ with $\bar\partial u=\alpha$, satisfying
the estimate
  $$||u||^2\leq\int_D|u|^2e^{-\varphi}d\lambda\leq 
     4\int_D|f|^2G^2(\chi'\circ G)^2e^{-\varphi}d\lambda.$$
For every $z\in D$ with $G(z)<(k-1)t_k$ there exists $j\leq k-1$
such that $G_j(z)<t_k$, and therefore $G_k\geq t_k$ on 
$\{G\leq(k-1)t_k-1\}$. It follows that $||u||<\infty$ and thus
$F:=f\chi\circ G-u\in L^2_h(D)$. Since $e^{-\varphi}$ is not
locally integrable near $w_1,\dots,w_k$, we conclude that 
$F(w_1)=\dots=F(w_{k-1})=0$ and $F(w_k)=f(w_k)$ (the latter
since $G(w_k)\geq (k-1)t_k$). 
\end{proof}

\section{Regularity of the volume of the Azukawa indicatrix}\label{section:regularity}
For $k\geq 1$ we define the \textit{$k$-th order Carath\'eodory-Reiffen pseudometric} as follows
\begin{equation}
\gamma_D^{(k)}(z;X):=\sup\left\{\left|f^{(k)}(z)X/k!\right |^{1/k}: f^{(j)}(z)=0,j=0,\ldots,k-1\right\},
\end{equation}
$z\in D,\; X\in\mathbb C^n$.

Recall that the bounded domain $D\subset\mathbb C^n$ is called \textit{strictly hyperconvex} if there are a bounded domain $\Omega\subset\mathbb C^n$, 
a continuous plurisubharmonic function $u:\Omega\to(-\infty,1)$ such that $D=\{u<0\}$, $u$ is exhaustive for $\Omega$ and for all $c\in[0,1]$ the set $\{u<c\}$ is connected (see \cite{Niv 1995}). 
It is elementary to see that $\gamma_D^{(k)}\leq A_D$. In general, the function $A_D$ is upper semicontinuous (see \cite{Jar-Pfl 1995}) and in the case of the hyperconvex $D$ even continuous (see \cite{Zwo 2000a}). 

It follows directly from the definition that the functions $D\times\mathbb C^n\owns (z;X)\to \gamma_D^{(k)}(z;X)$ are logarithmically plurisubharmonic. 
Recall that for a strictly hyperconvex $D$ and for any $z\in D$ we have the convergence $\lim\sb{k\to\infty}\gamma_D^{(k)}(z;X)=A_D(z;X)$ for almost all $X\in\mathbb C^n$ (Theorem 1 in \cite{Niv 1995}). Consequently, for strictly hyperconvex domain $D$ we get that the function
$D\times\mathbb C^n\owns(z;X)\to A_D(z;X)$ is logarithmically plurisubharmonic. Standard approximation properties allow us to deduce the following

\begin{proposition}\label{Azukawa-psh}
 Let $D$ be a pseudoconvex domain in $\mathbb C^n$. Then $\log A_D$ is plurisubharmonic (as a function defined on $D\times\mathbb C^n$).
\end{proposition}

For the pseudoconvex domain $D\subset\mathbb C^n$ define the following pseudoconvex (see e. g. \cite{Jar-Pfl 2000} and use the logarithmic plurisubharmonicity of $A_D$)
Hartogs domain with the basis $D$ and  balanced fibers
\begin{equation}
\Omega_D:=\{(z;X)\in D\times\mathbb C^n:A_D(z;X)<1\}.
\end{equation}
Consequently, making use of Theorem 1.4 from \cite{Ber 1998} ($\Omega_D(z)=I_D(z)$) we get the following result.

\begin{theorem}\label{vol-psh}
 Let $D$ be a pseudoconvex domain in $\mathbb C^n$ then the function
\begin{equation}
D\owns z\to -\log\lambda^{2n}(I_D(z))
\end{equation}
is plurisubharmonic.
\end{theorem}

It is natural to ask the question on the logarithmic convexity of $A_D$ in the case $D$ is convex. It turns out that the answer is positive.

\begin{theorem}\label{vol-convex}
 Let $D$ be a convex domain in $\mathbb C^n$. Then the function
\begin{equation}
D\owns z\to -\log\lambda^{2n}(I_D(z))
\end{equation}
is convex.
\end{theorem}
\begin{proof}
Due to the Lempert Theorem (see e. g. \cite{Lem 1981})) we have the equality $A_D=\kappa_D$, where $\kappa_D$ is the Kobayashi pseudometric of $D$. Without loss of generality we may assume that $D$ is bounded.
Let
$t\in[0,1]$, $w,z\in D$. We claim that $tI_D(w)+(1-t)I_D(z)\subset
I_D(tw+(1-t)z)$. Actually, let $X\in I_D(w)$, $Y\in I_D(z)$. Then
there are analytic discs $f,g:\mathbb D\to D$ such that $f(0)=w$, $g(0)=z$, $f^{\prime}(0)=X$,
$g^{\prime}(0)=Y$. Consequently, the mapping
$h:=tf+(1-t)g$ maps $\mathbb D$ into $D$, $h(0)=tw+(1-t)z$, $h^{\prime}(0)=tX+(1-t)Y$,
so $tX+(1-t)Y\in I_D(tw+(1-t)z)$.

It follows from the Brunn-Minkowski inequality that the Lebesgue measure is logarithmically concave (see e. g. \cite{Pre 1980}); therefore,
\begin{multline}
\lambda^{2n}(I_D(tw+(1-t)z)\geq \lambda^{2n}(tI_D(w)+(1-t)I_D(z))\geq\\
\lambda^{2n}(I_D(w))^t\lambda^{2n}(I_D(z))^{1-t}
\end{multline} 
which finishes the proof.
\end{proof}

%Recall that for a bounded pseodconvex domain $D$ in $\mathbb C^n$ we define the function $F_D$ by the formula
%$F_D(z):=\lambda^{2n}(I_D(z))K_D(z)$, $z\in D$.

%$F_D\geq 1$ (see \cite{Blo-Zwo 2015}) and 

The higher dimensional Suita conjecture (i. e. the inequality (\ref{equation:suita-conjecture})) may also be presented in the following way
\begin{equation}
F_D(w):=\root{n}\of{K_D(w)\cdot \lambda^{2n}(I_D(w))}\geq 1,\; w\in D.
\end{equation}
Note that the function $F_D$ has the following properties:
\begin{itemize}
\item $F$ is the biholomorphic invariant,\\
\item if $D$ is a bounded pseudoconvex balanced domain then $F_D(0)=1$.
\end{itemize}

Recall that for all non-trivial examples studied so far we have the property
$\lim\sb{z\to\partial D}F_D(z)=1$ (see \cite{Blo-Zwo 2015} and \cite{Blo-Zwo 2016}). Recently, the property was proven for the class of strongly pseudoconvex domains.

\begin{proposition}\label{strongly-psc} {\rm (see \cite{Bal-Bor-Mah-Ver 2018})}
 Let $D$ be a strongly pseudoconvex domain in $\mathbb C^n$. Then $\lim\sb{z\to\partial D}F_D(z)=1$.
\end{proposition}

Note that the above property follows directly from a recent result from \cite{Die-For-Wold 2014} (see Theorem 4.1 in \cite{Kim-Zhang 2014}).

%\begin{proposition}\label{strongly-psc}
% Let $D$ be a strongly pseudoconvex domain in $\mathbb C^n$. Then $\lim\sb{z\to\partial D}F_D(z)=1$.
%\end{proposition}

{\bf Remark} The above considerations let us formulate the question whether the convergence above holds in two following cases

\begin{itemize}
 \item $D$ is a bounded convex domain,
  \item $D$ is a bounded smooth pseudoconvex domain.
\end{itemize}


\begin{thebibliography}{Jar-Pfl 1995}

%\bibitem{Ahag-Czyz-Lun 2016} {\sc P. Ahag, R. Czy\.z, P. H. Lundow}, {\sl A counterexample to a conjecture by Blocki-Zwonek}, Exp. Math. 2016, to appear,  DOI: 10.1080/10586458.2016.1230913.

\bibitem{Bal-Bor-Mah-Ver 2018} {\sc G.P. Balakumar, D. Borah, P. Mahajan, K. Verma}, {\sl Remarks on the higher dimensional Suita conjecture}, 
https://arxiv.org/abs/1808.08007.

\bibitem{Ber 1998} {\sc B. Berndtsson}, {\sl Prekopa's theorem and Kiselman's minimum principle for plurisubharmonic
functions}, Math. Ann. (1998), 312(4), 785--792.

\bibitem{Blo 2014a} {\sc Z. B\l ocki}, {\sl A lower bound for the Bergman kernel and the Bourgain-Milman inequality}, 
Geometric Aspects of Functional Analysis, Israel Seminar (GAFA) 2011--2013, eds. B. Klartag, E. Milman, Lecture Notes in Mathematics 2116, pp. 53--63, Springer, 2014.

\bibitem{Blo 2014b} {\sc Z. B\l ocki}, {\sl Cauchy-Riemann meet Monge-Amp\'ere}, Bulletin of Mathematical Sciences, 4 (2014), 433--480.

\bibitem{Blo 2015} {\sc Z. B\l ocki}, {\sl Bergman kernel and pluripotential theory}, in Analysis, complex geometry, and mathematical physics: in honor of Duong H. Phong, eds. P. M. N. Feehan, J. Song, B. Weinkove, R. A. Wentworth, Contemporary Mathematics 644, pp. 1-10, American Mathematical Society, 2015.

\bibitem{Blo-Zwo 2015} {\sc Z. B\l ocki, W. Zwonek}, {\sl Estimates for the Bergman kernel and the multidimensional Suita conjecture}, New York J. Math., 21 (2015), 151--161.

\bibitem{Blo-Zwo 2016} {\sc Z. B\l ocki, W. Zwonek}, {\sl Suita conjecture for some convex ellipsoids in $\mathbb C^2$}, Exp. Math.,  25(1) (2016), 8--16.

\bibitem{Blo-Zwo 2018} {\sc Z. B\l ocki, W. Zwonek}, {\sl One dimensional estimates for the Bergman kernel and logarithmic capacity}, Proc. Amer. Math. Soc., 146 (2018), 2489--2495.

\bibitem{Die-For-Wold 2014} {\sc K. Diederich, J. F. Fornaess, E. F. Wold}, {\sl Exposing Points on the Boundary of a Strictly Pseudoconvex or a Locally Convexifiable Domain of Finite $1$-Type}
J. Geometric Analysis, (2014), 24, Issue 4, 2124--2134.

\bibitem{Don-Fef 1983} {\sc H. Donnelly, C. Fefferman},  {\sl $L^2$-cohomology and index theorem for the Bergman
metric}, Ann. of Math. 118 (1983), 593--618.

%\bibitem{For 2016} {\sc J. E. Fornaess}, {\sl On a conjecture by Blocki and Zwonek}, Sci. China Math. (2016), to appear. doi:10.1007/s11425-016-5124-7.

\bibitem{Jar-Pfl 1995} {\sc M. Jarnicki, P. Pflug}, {\sl Remarks on the pluricomplex Green function}
Indiana Univ. Math. J. 44 (1995), no. 2, 535--543.

\bibitem{Jar-Pfl 2000} {\sc M. Jarnicki, P. Pflug}, {\sl Extension of holomorphic functions}, de Gruyter Expositions in Mathematics, 
34. Walter de Gruyter \& Co., Berlin, 2000. x+487 pp.

\bibitem{Lem 1981} {\sc L. Lempert}, {\sl La m\'etrique de Kobayashi et la repr\'esentation des domaines sur la boule},  Bull. Soc. Math. France 
109 (1981), 427--474.

\bibitem{Kim-Zhang 2014} {\sc K.-T. Kim, L. Zhang}, {\sl On the uniform squeezing property and the squeezing function}, https://arxiv.org/pdf/1306.2390.pdf.

\bibitem{Niv 1995} {\sc S. Nivoche}, {\sl The Pluricomplex Green Function, Capacitative Notions, and Approximation Problems in $\mathbb C^n$},
Indiana Univ. Math. J., 44, no. 2, 1995, 489--510.

\bibitem{Pfl-Zwo 2017} {\sc P. Pflug, W. Zwonek}, {\sl $L_h^2$-Functions in unbounded balanced domains},  J. Geometric Analysis, 27(3) (2017), 2118--2130.

\bibitem{Pre 1980} {\sc A. Prekopa}, {\sl Logarithmic concave measures and related topics}, Stochastic programming (Proc. Internat. Conf., Univ. Oxford, Oxford, 1974). London-New York: Academic Press. pp. 63--82.

\bibitem{Wie 1984} {\sc J. Wiegerinck}, {\sl Domains with finite-dimensional Bergman space}, Math. Z. 187(4), 559--562 (1984).

\bibitem{Zwo 2000a} {\sc W. Zwonek}, {\sl Regularity properties of the Azukawa metric}, J. Math. Soc. Japan 52 (2000), no. 4, 899--914.

\bibitem{Zwo 2000b} {\sc W. Zwonek}, {\sl Completeness, Reinhardt domains and the method of complex geodesics in the theory of invariant functions},
Dissertationes Math. 388 (2000), 103 pp.

%\bibitem{Zwo 2010} {\sc W. Zwonek}, {\sl Asymptotic behavior of the sectional curvature of the Bergman metric for annuli}, Ann. Polon. Math. 98 (2010), 291--299. 
\end{thebibliography}
\end{document}